\documentclass{amsart}

\usepackage{amsmath,amssymb,gensymb}
\usepackage[a4paper,left=2.25cm,right=2.25cm,top=3cm,bottom=3cm]{geometry}
\usepackage{tikz}\usetikzlibrary{decorations.pathmorphing,decorations.markings}
\usepackage{hyperref}
\usepackage{comment}
\usepackage{color}

\allowdisplaybreaks[1]

\newcommand{\bra}[1]{\left\langle #1\right|}
\newcommand{\ket}[1]{\left|#1\right\rangle}

\renewcommand{\b}[1]{\bar{#1}}
\renewcommand{\leq}{\leqslant}
\renewcommand{\geq}{\geqslant}

\def\phid{\phi^\dagger}

\newtheorem*{defn}{Definition}
\newtheorem*{lem}{Lemma}
\newtheorem*{thm}{Theorem}

\title[A summation formula for Macdonald polynomials]
{A summation formula for Macdonald polynomials}

\author{Jan~de~Gier}

\author{Michael~Wheeler}

\address{School of Mathematics and Statistics, University of Melbourne, Parkville, Victoria 3010, Australia}
\email{jdgier@unimelb.edu.au, wheelerm@unimelb.edu.au}

\begin{document}

%abstract
\begin{abstract}
We derive an explicit sum formula for symmetric Macdonald polynomials. Our expression contains multiple sums over the symmetric group and uses the action of Hecke generators on the ring of polynomials. In the special cases $t=1$ and $q=0$, we recover known expressions for the monomial symmetric and Hall--Littlewood polynomials, respectively. Other specializations of our formula give new expressions for the Jack and $q$--Whittaker polynomials.
\end{abstract}

\maketitle

\section{Introduction}

Symmetric Macdonald polynomials \cite{macd88,MacdBook} are a family of multivariable orthogonal polynomials indexed by partitions, whose coefficients depend rationally on two parameters $q$ and $t$. In the case $q=t$ they degenerate to the celebrated Schur polynomials, which are central in the representation theory of both the general linear and symmetric groups. Using the notation $m_\lambda$ to denote the monomial symmetric polynomial indexed by a partition $\lambda$, the standard definition of Macdonald polynomials in the literature is the following.
\begin{defn}
The Macdonald polynomial $P_{\lambda}(x_1,\dots,x_n;q,t)$ is the unique homogeneous symmetric polynomial in $(x_1,\dots,x_n)$ which satisfies
\begin{align*}
P_{\lambda}(x_1,\dots,x_n;q,t)
=
m_{\lambda}(x_1,\dots,x_n)
+
\sum_{\mu < \lambda}
c_{\lambda,\mu}(q,t)
m_{\mu}(x_1,\dots,x_n),
\qquad
\langle P_{\lambda}, P_{\mu} \rangle = 0, \ \lambda \neq \mu,
\end{align*}
with respect to the Macdonald inner product on power sum symmetric functions (\cite{MacdBook}, Chapter VI, Equation (1.5)),
%$\langle p_{\lambda}, p_{\mu} \rangle = \delta_{\lambda,\mu} z_{\lambda}(q,t)$, 
and where $<$ denotes the dominance order on partitions (\cite{MacdBook}, Chapter I, Section 1).
\end{defn}

Up to normalization, Macdonald polynomials can alternatively be defined as the unique eigenfunctions of certain linear difference operators acting on the space of all symmetric polynomials \cite{MacdBook}. They can also be expressed combinatorially as multivariable generating functions \cite{HaglundHL1,HaglundHL2,RamY}.

Opdam \cite{Opdam} and Cherednik \cite{Cher95a,Cher95b} generalized the earlier theory of Macdonald to a non-symmetric setting, and defined non-symmetric Macdonald polynomials $E_\mu$ that are indexed by compositions $\mu$. Non-symmetric Macdonald polynomials are defined as joint eigenfunctions of a family of commuting operators in the double affine Hecke algebra \cite{Cher95a,Cher95b}, and the symmetric polynomial $P_\lambda$ is obtained as the sum over all 
$E_\mu$ whose index $\mu$ lies in the orbit of $\lambda$ under the symmetric group. 

The purpose of this article is to state and prove an explicit sum formula for calculating Macdonald polynomials. This result is derived directly from the recent work \cite{CantinidGW}.

\section{Notation}

\subsection{Partitions and compositions}
\label{sec:partitions}
 
Fix a partition $\lambda=(\lambda_1,\dots,\lambda_n)$, where $\lambda_1 \geq \cdots \geq \lambda_n \geq 0$. Let $\ell(\lambda)$ denote the length of $\lambda$, which is the number of non-trivial parts $\lambda_i > 0$ in $\lambda$. We associate a monomial $x_{\lambda}$ in the variables $(x_1,\dots,x_n)$ to any partition, which is simply $x_{\lambda} = \prod_{i=1}^{\ell(\lambda)} x_i$. 

We write $\lambda' = (\lambda'_1,\dots,\lambda'_r)$ for the partition conjugate to $\lambda$, as is standard in the literature.

Let $\lambda$ be a partition whose largest part is $\lambda_1 = r$. For all $0 \leq k \leq r$, we define a partition $\lambda[k]$ which is obtained by replacing all parts in $\lambda$ of size $\leq k$ with 0. For example, for $\lambda = (3,3,2,1,1,1,0)$ we have
\begin{align*}
\lambda[0] = (3,3,2,1,1,1,0),
\quad
\lambda[1] = (3,3,2,0,0,0,0),
\quad
\lambda[2] = (3,3,0,0,0,0,0),
\quad
\lambda[3] = (0,0,0,0,0,0,0),
\end{align*}
and in general $\lambda[0] = \lambda$, $\lambda[r] = 0$.

We will also consider compositions $\mu = (\mu_1,\dots,\mu_n)$ whose parts satisfy 
$\mu_i \geq 0$, but are not ordered in any particular way. We write $\mu^{+}$ to denote the unique partition obtained by arranging the parts of $\mu$ in weakly decreasing order. For 
$\mu$ an arbitrary composition, the part-multiplicity function $m_i(\mu)$\footnote{Not to be confused with the standard notation $m_\lambda$ for monomial symmetric polynomials.} is defined as follows:
\begin{align*}
m_i(\mu)
=
\#\{ \mu_k : \mu_k = i \}.
\end{align*}
Further to this, two functions which take a pair of compositions as arguments will be used in our main formula: 
\begin{align*}
&
a_i(\lambda,\mu)
=
\#\{ (\lambda_k,\mu_k): \lambda_k = 0,\ \mu_k = i \},
\quad
b_i(\lambda,\mu)
=
\#\{ (\lambda_k,\mu_k): i = \lambda_k > \mu_k \}.
\end{align*}
For example, for $\lambda = (4,4,3,3,3,2,0,0,0)$ and $\mu = (0,3,0,0,3,0,4,3,4)$, we have
\begin{align*}
a_i(\lambda,\mu) 
\equiv
a_i\left(
\begin{array}{ccccccccc}
4 & 4 & 3 & 3 & 3 & 2 & 0 & 0 & 0
\\
0 & 3 & 0 & 0 & 3 & 0 & 4 & 3 & 4
\end{array}
\right)
=
\left\{
\begin{array}{ll}
0, & i=1
\\
0, & i=2
\\
1, & i=3
\\
2, & i=4
\end{array}
\right.
\end{align*}
\begin{align*}
b_i(\lambda,\mu) 
\equiv
b_i\left(
\begin{array}{ccccccccc}
4 & 4 & 3 & 3 & 3 & 2 & 0 & 0 & 0
\\
0 & 3 & 0 & 0 & 3 & 0 & 4 & 3 & 4
\end{array}
\right)
=
\left\{
\begin{array}{ll}
0, & i=1
\\
1, & i=2
\\
2, & i=3
\\
2, & i=4
\end{array}
\right.
\end{align*}

\subsection{Permutations}

Let $S_n$ be the group of all permutations of $(1,\dots,n)$. For any partition $\lambda$, let $S_{\lambda}$ be the quotient of $S_n$ by the subgroup $S^{\lambda}_n$ of permutations which leave $\lambda$ invariant. $S_{\lambda} = S_n/S^{\lambda}_n$ is thus the set of all permutations which have a distinct action on $\lambda$. For example, for $\lambda = (2,2,0,0)$ we find
\begin{align*}
S_{\lambda}
=
\{ (1,2,3,4),\ (1,3,2,4),\ (1,3,4,2),\ (3,1,2,4),\ (3,1,4,2),\ (3,4,1,2) \}
\end{align*}
as the complete set of permutations with a distinguishable action on $\lambda$.

\subsection{Hecke algebra}

We consider polynomial representations of the Hecke algebra of type $A_{n-1}$, with generators $T_{i}$ given by 
\begin{align}
\label{hecke-gen}
T_i
=
t-\frac{tx_i-x_{i+1}}{x_i-x_{i+1}}(1-s_i),
\qquad
1 \leq i \leq n-1,
\end{align} 
where $s_i$ is the transposition operator with action $s_i f(\dots,x_i,x_{i+1},\dots) = f(\dots,x_{i+1},x_i,\dots)$ on functions in $(x_1,\dots,x_n)$. It can be verified that the operators \eqref{hecke-gen} indeed give a faithful representation of the Hecke algebra:
\begin{align}
\label{hecke-alg}
(T_{i}-t)(T_{i}+1)=0,
\qquad
T_{i} T_{i\pm 1} T_{i} = T_{i \pm 1} T_i T_{i\pm 1},
\qquad
T_i T_j = T_j T_i, \ |i-j| > 1.
\end{align}
In view of the relations for the generators, we can define $T_{\sigma}$ unambiguously as any product of simple transpositions $T_{i}$ which gives the permutation $\sigma$. For example, for $\sigma = (3,2,1,4)$ we can write
\begin{align*}
T_{\sigma}
=
T_{2} T_{1} T_{2}
=
T_{1} T_{2} T_{1},
\end{align*}
with both expressions having an equivalent action on the ring of polynomials in 
$(x_1,\dots,x_n)$.

\section{Main formula}

\begin{thm}
Let $\lambda = (\lambda_1,\dots,\lambda_n)$ be a partition with largest part 
$\lambda_1 = r$, and from this define a set of partitions $\lambda[0],\dots,\lambda[r]$ as in Section \ref{sec:partitions}. Then the Macdonald polynomial $P_{\lambda}$ can be written in the form
\begin{align}
\label{main-formula}
P_{\lambda}(x_1,\dots,x_n;q,t)
=
\sum_{\sigma \in S_{\lambda} }
T_{\sigma}
\circ
x_{\lambda}
\circ
\prod_{i=1}^{r-1}
\left(
\sum_{\sigma \in S_{\lambda[i]}}
C_i
\left(
\begin{array}{@{}c@{}}
\lambda[i-1]
\medskip \\ 
\sigma\circ\lambda[i]
\end{array}
\right)
T_{\sigma}
\circ
x_{\lambda[i]}
\circ
\right)
1
\end{align}
with coefficients\footnote{We will use three notations for the coefficients interchangeably:
$
C_i( \lambda, \mu) 
\equiv
C_i\left(\begin{array}{@{}c@{}} \lambda \\ \mu \end{array}\right)
\equiv
C_i\left(
\begin{array}{@{}c@{}c@{}c@{}} \lambda_1 & \cdots & \lambda_n 
\\ 
\mu_1 & \cdots & \mu_n \end{array}
\right)
$.} that satisfy $C_i( \lambda, \mu) = 0$ if any $0 < \lambda_k < \mu_k$, and 
\begin{align}
\label{coeff}
C_i(\lambda,\mu)
\equiv
C_i\left(\begin{array}{@{}c@{}c@{}c@{}} 
\lambda_1 & \cdots & \lambda_n \\ 
\mu_1 & \cdots & \mu_n \end{array}\right)
=
\prod_{j = i+1}^{r}
\left(
q^{(j-i)a_j(\lambda,\mu)}
\prod_{k=1}^{b_j(\lambda,\mu)}
\frac{1-t^k}{1-q^{j-i}t^{\lambda'_i-\lambda'_j+k}}
\right),
\end{align}
otherwise.
\end{thm}

In order to clarify the structure of the expression \eqref{main-formula}, we give explicit examples in the Appendix, where we calculate some Macdonald polynomials of sufficiently small size. We give its proof in Section \ref{sec:proof}.

We wish to point out that \eqref{main-formula} has many structural features in common with the work of Kirillov and Noumi \cite{Kirillov_Noumi1,Kirillov_Noumi2}. In these papers the authors construct families of raising operators, which act on Macdonald polynomials by adding columns to the indexing Young diagram. In \cite{Kirillov_Noumi1} the raising operators have an analogous form to Macdonald $q$-difference operators, while in \cite{Kirillov_Noumi2} the raising operators are constructed in terms of generators of the affine Hecke algebra. In both papers the Macdonald polynomial is obtained by the successive action of such raising operators on $1$, the initial state. It would be very interesting to find a precise connection between the results of \cite{Kirillov_Noumi1,Kirillov_Noumi2} and our formula \eqref{main-formula}, if one exists.

\section{Proof}
\label{sec:proof}

In this section we sketch the proof of \eqref{main-formula}, citing results from our earlier paper \cite{CantinidGW}, where we obtained a matrix product formula for a family of non-symmetric polynomials $f_{\lambda}(x_1,\dots,x_n;q,t)$. These polynomials are indexed by compositions $\lambda$, and satisfy the following relations with the generators of the Hecke algebra:
\begin{align*}
T_i f_{\lambda_1,\dots,\lambda_n}(x_1,\dots,x_n;q,t)
=
f_{\lambda_1,\dots,\lambda_{i+1},\lambda_i,\dots,\lambda_n}
(x_1,\dots,x_n;q,t),
\qquad
\text{when}\ \ 
\lambda_i > \lambda_{i+1}.
\end{align*}
These relations allow us to express a polynomial $f_{\mu}$ indexed by an arbitrary composition $\mu$ in terms of Hecke generators acting on $f_{\mu^{+}}$. As was demonstrated in \cite{CantinidGW}, the symmetric Macdonald polynomial $P_{\lambda}$ is obtained as a sum over all polynomials $f_{\mu}$, whose composition is a distinct permutation of the partition $\lambda$:
\begin{align}
\label{mac-perms}
P_{\lambda}(x_1,\dots,x_n;q,t)
=
\sum_{\sigma \in S_{\lambda}}
f_{\sigma \circ \lambda}
(x_1,\dots,x_n;q,t)
=
\sum_{\sigma \in S_{\lambda}}
T_{\sigma}
\circ
f_{\lambda}(x_1,\dots,x_n;q,t).
\end{align}
We now recall the matrix product expression for $f_{\lambda}$ as obtained in \cite{CantinidGW}. Assume $\lambda\subseteq r^n$ and introduce the following family of $(r-s+2) \times (r-s+1)$ operator-valued matrices $L^{(s)}(x)$, $1 \leq s \leq r$. We index rows by $i \in \{0,s,\dots,r\}$ and columns by $j \in \{0,s+1,\dots,r\}$\footnote{This unusual indexing of the entries of the matrix is the most convenient for our purposes, since we ultimately want to identify these matrix elements with the partitions $\lambda[s]$ introduced in Section \ref{sec:partitions}.}, and take
\begin{align}
\nonumber
L^{(s)}_{00}
=
1,
\qquad
&
L^{(s)}_{0j}
=
\phi_j,\ \ 
\text{for }\ s+1 \leq j \leq r,
\qquad
L^{(s)}_{i0}(x)
=
x \times
\left\{
\begin{array}{ll}
\prod_{l=i+1}^{r}
k_l,
&
i=s
\\
\\
\phi^{\dagger}_i
\prod_{l=i+1}^{r}
k_l,
&
s < i \leq r
\end{array}
\right.
\\
\nonumber
\\
\label{Lmat}
&
\qquad
L^{(s)}_{ij}(x)
=
x \times
\left\{
\begin{array}{ll}
\prod_{l=i+1}^{r}
k_l,
&
i=j
\\
\\
\phi^{\dagger}_i
\phi_j
\prod_{l=i+1}^{r}
k_l,
&
i>j
\\
\\
0,
&
i<j
\end{array}
\right.
\quad\ \text{for }\ 
s \leq i \leq r,\ 
s+1 \leq j \leq r.
\end{align}
We also introduce a twist operator 
\begin{align}
\label{twist}
S^{(s)}=\prod_{l=s+1}^{r} k_l^{(l-s)u}
\ \ \text{for }\
1 \leq s \leq r-1,
\qquad
S^{(r)} = 1,
\end{align} 
where we perform the re-parametrization $q=t^{u}$. The operators 
$\{k,\phi, \phi^{\dagger}\}$ are generators of the $t$-boson algebra:
\begin{align}
\label{t-bos}
\phi k = t k \phi,
\qquad
t \phid k = k \phid,
\qquad
\phi \phid - t \phid \phi = 1-t,
\end{align}
with subscripts used to denote commuting copies of the algebra. Note that all operators in $L^{(s)}(x)$ implicitly also carry an index $(s)$, as we will assume that operators in $L^{(s)}(x)$ and $L^{(s')}(x)$ (as well as $S^{(s)}$ and $S^{(s')}$) commute for $s \not= s'$.

From this we construct an $(r+1)$-dimensional vector $\mathbb{A}(x)$ and a total twist operator $\mathbb{S}$:
\begin{align*}
\mathbb{A}(x)
=
L^{(1)}(x) \dots L^{(r)}(x),
\qquad
\mathbb{S}
=
S^{(1)} \dots S^{(r)}.
\end{align*}
We denote the components of $\mathbb{A}(x)$ by 
$A_i(x)$. The principal result of \cite{CantinidGW} was the formula
\begin{align}
\label{matrix-product}
\Omega_{\lambda^{+}}(q,t)
f_{\lambda}(x_1,\dots,x_n;q,t)
=
{\rm Tr}\left[
A_{\lambda_1}(x_1)
\dots
A_{\lambda_n}(x_n)
\mathbb{S}
\right],
\end{align}
where $\lambda = (\lambda_1,\dots,\lambda_n)$ is any composition consisting of $n$ parts 
$\lambda_i \geq 0$, $\Omega_{\lambda^{+}}(q,t)$ is an overall normalization which only depends on the partition $\lambda^+$. Here we trace over suitable representations of the (multiple copies of the) $t$-boson algebra 
\eqref{t-bos}. Up to changes of the normalization $\Omega_{\lambda^{+}}$, equation \eqref{matrix-product} is independent of the choice of these representations as long as they are faithful, so we choose the Fock representation for all copies of the $t$-boson algebra:
\begin{align}
\label{fock}
\phi \ket{m} = (1-t^m) \ket{m-1},
\qquad
\phid \ket{m} = \ket{m+1},
\qquad
k \ket{m} = t^m \ket{m},
\\
\nonumber
\bra{m} \phi = (1-t^{m+1}) \bra{m+1},
\qquad
\bra{m} \phid = \bra{m-1},
\qquad
\bra{m} k = t^m \bra{m}, 
\end{align}
for which the correct normalization\footnote{This normalization is chosen such that $f_{\lambda}$ is monic, or in other words, the coefficient of its leading monomial $x_1^{\lambda_1}\dots x_n^{\lambda_n}$ is 1. This is clearly the normalization required for \eqref{mac-perms} to be valid.} in \eqref{matrix-product} is
\begin{align}
\label{norm}
\Omega_{\lambda^{+}}(q,t)
= 
\prod_{i=1}^{r}
\prod_{j=i+1}^{r}
\frac{1}{1-q^{j-i} t^{(\lambda^{+})'_i-(\lambda^{+})'_j}},
\end{align}
where $r$ is the largest part of $\lambda$. 

A useful way of decomposing \eqref{matrix-product}, which is central to our proof of \eqref{main-formula}, is in terms of the transition matrix elements
\begin{align*}
T^{(s)}_{\lambda,\mu}(x_1,\dots,x_n)
=
{\rm Tr}\left[
L^{(s)}_{\lambda_1,\mu_1}(x_1)
\dots
L^{(s)}_{\lambda_n,\mu_n}(x_n)
S^{(s)}
\right],
\end{align*}
where $\lambda$ and $\mu$ are compositions taking values in $\{0,s,\dots,r\}$ and 
$\{0,s+1,\dots,r\}$, respectively. Using this definition we are able to write down the recursion relation
\begin{align}
\label{rec-rel}
\Omega^{(s)}_{\lambda^{+}}(q,t)
f^{(s)}_{\lambda}(x_1,\dots,x_n;q,t)
=
\sum_{\mu} 
T^{(s)}_{\lambda,\mu}(x_1,\dots,x_n)
\Omega^{(s+1)}_{\mu^{+}}(q,t)
f^{(s+1)}_{\mu}(x_1,\dots,x_n;q,t),
\end{align}
where the sum is taken over compositions $\mu$ such that $m_i(\lambda) = m_i(\mu)$ for all $s+1 \leq i \leq r$. Here $f^{(s)}_{\lambda}(x_1,\dots,x_n;q,t)$ denotes a reduced version of \eqref{matrix-product}, in which the operators $A_i(x)$ are components of 
$\mathbb{A}^{(s)}(x) \equiv L^{(s)}(x) \dots L^{(r)}(x)$, the twist is given by 
$\mathbb{S}^{(s)} \equiv S^{(s)} \dots S^{(r)}$, and
\begin{align*}
\Omega^{(s)}_{\lambda^{+}}(q,t)
= 
\prod_{i=s}^{r}
\prod_{j=i+1}^{r}
\frac{1}{1-q^{j-i} t^{(\lambda^{+})'_i-(\lambda^{+})'_j}}.
\end{align*}
To complete the proof of \eqref{main-formula}, we establish the following result.

\begin{lem}
Let $\lambda$ be a partition with parts in $\{0,s,\dots,r\}$, and $\mu$ a composition with parts in 
$\{0,s+1,\dots,r\}$, such that $m_i(\lambda) = m_i(\mu)$ for all $s+1 \leq i \leq r$. Then 
\begin{align}
\label{trans-coeff}
T^{(s)}_{\lambda,\mu}(x_1,\dots,x_n)
=
\frac{
x_{\lambda}
C_s(\lambda,\mu)
\Omega^{(s)}_{\lambda}(q,t)
}
{
\Omega^{(s+1)}_{\mu^{+}}(q,t)
},
\end{align}
where $C_s(\lambda,\mu)$ is given by \eqref{coeff}.
\end{lem}

\begin{proof}
We begin by noticing that, by virtue of the vanishing of some of the matrix elements \eqref{Lmat}, $T^{(s)}_{\lambda,\mu} = 0$ if $0 < \lambda_k < \mu_k$ for some $k$. Hence both sides of \eqref{trans-coeff} have the same vanishing property. This constrains $\lambda_k \geq \mu_k$ for all $1 \leq k \leq \ell(\lambda)$, which we assume from this point on. Given that $\lambda$ is a partition, using the explicit form of the $L^{(s)}$ matrix entries we find that
\begin{multline}
\label{trace-initial}
T^{(s)}_{\lambda,\mu}(x_1,\dots,x_n)
=
\prod_{i=1}^{\ell(\lambda)}
x_i
\times
{\rm Tr}
\left[
\left(\prod_{l=s+1}^{r} k_l^{m_s(\lambda)} \right)
\left( \prod_{j=s+1}^{r} \phi_j^{a_j(\lambda,\mu)} \right)
S^{(s)}
\left( (\phid_r)^{b_r(\lambda,\mu)} \prod_{j=s+1}^{r-1} \phi_j^{\#} \right)
\right.
\\
\times
\left.
\left(k_r^{m_{r-1}(\lambda)} 
(\phid_{r-1})^{b_{r-1}(\lambda,\mu)} \prod_{j=s+1}^{r-2} \phi_j^{\#} \right)
\dots
\left( 
\prod_{l=s+2}^{r} k_l^{m_{s+1}(\lambda)}  
(\phid_{s+1})^{b_{s+1}(\lambda,\mu)} \right)
\right]
\end{multline}
where we use $\#$ to denote a multiplicity in the composition $\mu$ which does not affect the final answer, so we suppress it. Note that we have used the cyclicity of the trace to reorder the product, so that the leftmost parentheses correspond with parts of $\lambda$ of size $s$, the next parentheses with parts of size 0, followed by the twist, followed by parts of size $r$ and decreasing thereafter down to parts of size $s+1$.

The expression \eqref{trace-initial} is seemingly very complicated. But in view of the fact that operators with different subscripts commute, and the simple commutation \eqref{t-bos} between operators from the same copy of the algebra, as well as the factorized form 
\eqref{twist} of the twist, we can factorize \eqref{trace-initial} as follows:
\begin{align}
\label{trace-factor}
T^{(s)}_{\lambda,\mu}(x_1,\dots,x_n)
=
\prod_{i=1}^{\ell(\lambda)}
x_i
\times
\prod_{j=s+1}^{r}
q^{(j-s) a_j(\lambda,\mu)}
{\rm Tr}
\left[
\phi_j^{b_j(\lambda,\mu)}
(\phid_j)^{b_j(\lambda,\mu)}
k_j^{m_s(\lambda)+\cdots+m_{j-1}(\lambda)+(j-s)u}
\right].
\end{align} 
Each trace in this product can be calculated using the general relation
\begin{align*}
{\rm Tr}
\left[
\phi^b (\phid)^c k^d
\right]
=
\frac{\delta_{b,c} \prod_{i=1}^{b} (1-t^i)}{\prod_{i=0}^{b} (1-t^{d+i})},
\ \ 
\text{for }\ b,c \geq 0,\ d \in \mathbb{C},
\end{align*}
valid for the representation \eqref{fock} which we have used. Applying this relation to \eqref{trace-factor} and using the fact that $\sum_{s \leq i < j} m_i(\lambda) = \lambda'_s-\lambda'_j$, we obtain
\begin{align}
\label{prod_with_k=0}
T^{(s)}_{\lambda,\mu}(x_1,\dots,x_n)
=
x_{\lambda}
\times
\prod_{j=s+1}^{r}
\left(
\frac{q^{(j-s) a_j(\lambda,\mu)}\displaystyle{\prod_{k=1}^{b_j(\lambda,\mu)} (1-t^k)}}
{\displaystyle{\prod_{k=0}^{b_j(\lambda,\mu)} (1-q^{j-s} t^{\lambda'_s-\lambda'_j+k})}}
\right).
\end{align}
Finally, in view of the fact that $m_i(\lambda) = m_i(\mu)$ for all $s+1 \leq i \leq r$, we find that 
\begin{align*}
\Omega^{(s)}_{\lambda}(q,t)
\Big/
\Omega^{(s+1)}_{\mu^{+}}(q,t)
=
\prod_{j=s+1}^{r}
\frac{1}{1-q^{j-s} t^{\lambda'_s-\lambda'_j}},
\end{align*}
allowing us to extract the $k=0$ term from the denominator of \eqref{prod_with_k=0}, recovering exactly \eqref{trans-coeff}.
 
\end{proof}

Returning to equation \eqref{rec-rel}, we have thus shown that when $\lambda$ is a partition,
\begin{align}
\label{rec-rel2}
f^{(s)}_{\lambda}(x_1,\dots,x_n;q,t)
=
x_{\lambda} 
\times
\sum_{\mu} 
C_s(\lambda,\mu)
f^{(s+1)}_{\mu}(x_1,\dots,x_n;q,t),
\qquad
\forall\ 1 \leq s \leq r-1.
\end{align}
To complete the proof of equation \eqref{main-formula} it remains only to combine \eqref{mac-perms} with the recursion relation \eqref{rec-rel2}, iterated $r-1$ times, with the final value $f^{(r)}_{\mu} = x_{\mu}$. At each step in the iteration, we need to order the composition $\mu$ indexing $f^{(s+1)}_{\mu}$ so that it becomes a partition (otherwise the result \eqref{rec-rel2} is not valid at the next stage). This ordering is achieved by the action of $T_{\sigma}$ operators at each step.

\section{Specializations}

\subsection{Monomial symmetric polynomials}

By specializing $t=1$, the Macdonald polynomial $P_{\lambda}$ reduces to a monomial symmetric polynomial $m_{\lambda}(x_1,\dots,x_n)$. It is especially simple to evaluate \eqref{main-formula} at $t=1$. Due to the factors $1-t^k$ in \eqref{coeff}, we find that each coefficient $C_i$ vanishes, unless $b_j(\lambda[i-1],\sigma\circ\lambda[i])=0$ for all $i+1 \leq j \leq r$. This constrains each $\sigma$ (apart from that of the leftmost sum) to be the identity. Furthermore, because $a_j(\lambda[i-1],\lambda[i]) = 0$ for all $i+1 \leq j \leq r$, all coefficients $C_i(\lambda[i-1],\lambda[i]) = 1$.

Finally, at $t=1$ the Hecke generators \eqref{hecke-gen} reduce to simple transpositions $s_i$. Putting all of this together, equation \eqref{main-formula} reduces to
\begin{align*}
P_{\lambda}(x_1,\dots,x_n;q,1)
=
\sum_{\sigma \in S_{\lambda}}
s_{\sigma} \circ x_{\lambda} \circ \prod_{i=1}^{r-1} x_{\lambda[i]}
=
\sum_{\sigma \in S_{\lambda}}
\sigma
\circ
\left(
\prod_{i=1}^{n}
x_i^{\lambda_i}
\right)
=
m_{\lambda}(x_1,\dots,x_n),
\end{align*}
as required.

\subsection{Hall--Littlewood polynomials}

The next specialization of interest is when $q=0$, which gives the Hall--Littlewood polynomial $P_{\lambda}(x_1,\dots,x_n;t)$. In fact the formula \eqref{main-formula} simplifies greatly when $q=0$. Clearly
\begin{align*}
\left.
C_i
\left(
\begin{array}{@{}c@{}}
\lambda[i-1]
\medskip \\ 
\sigma\circ\lambda[i]
\end{array}
\right)
\right|_{q=0}
=
0,
\ \
\text{unless }\
a_j\left(\lambda[i-1],\sigma \circ \lambda[i]\right)=0
\ \ \text{for all } j>i.
\end{align*}
This constrains each permutation $\sigma$ (apart from that of the leftmost sum) to be the identity. Furthermore, in view of the fact that 
$b_j(\lambda[i-1],\lambda[i])=0$ for all $j>i$, we have 
$C_i(\lambda[i-1],\lambda[i])=1$ for all $1 \leq i \leq r-1$. Hence the formula 
\eqref{main-formula} reduces to
\begin{align}
\label{HL1}
P_{\lambda}(x_1,\dots,x_n;t)
=
\sum_{\sigma \in S_{\lambda}}
T_{\sigma}
\circ
x_{\lambda}
\circ
\prod_{i=1}^{r-1}
x_{\lambda[i]}
=
\sum_{\sigma \in S_{\lambda}}
T_{\sigma}
\circ
\left(
\prod_{i=1}^{n} x_i^{\lambda_i}
\right).
\end{align}
Equation \eqref{HL1} is a known expression for Hall--Littlewood polynomials, and seems to have been first pointed out in \cite{Duchamp_et_al} (Remark 3, following Theorem 3.1). For more details we refer the reader to Section 6.2 of \cite{Hivert}, where the correspondence between \eqref{HL1} and the standard sum expression
\begin{align}
\label{HL2}
P_{\lambda}(x_1,\dots,x_n;t)
=
\sum_{\sigma \in S_{\lambda}}
\sigma
\circ
\left(
\prod_{i=1}^{n} x_i^{\lambda_i}
\prod_{\lambda_i > \lambda_j}
\frac{x_i-t x_j}{x_i-x_j}
\right)
\end{align}
(\cite{MacdBook}, Chapter III, Equation (2.2)) is explained in greater detail.

\subsection{Jack polynomials}

The Jack polynomial $P_{\lambda}^{(\alpha)}(x_1,\dots,x_n)$ is obtained as the 
$q=t^{\alpha}, t \rightarrow 1$ limit of the Macdonald polynomial $P_{\lambda}$ (\cite{MacdBook}, Chapter VI, Section 10). Let us first examine the effect of this limit on the coefficients \eqref{coeff}. We find that
\begin{align*}
C_i(\lambda,\mu) = 0,\ \text{if any}\ 0 < \lambda_k < \mu_k,
\quad
C_i(\lambda,\mu)
=
\prod_{j=i+1}^{r}
\prod_{k=1}^{b_j(\lambda,\mu)}
\left(
\frac{k}{(j-i)\alpha+\lambda_i'-\lambda_j'+k}
\right),\ \text{otherwise.}
\end{align*}
Further to this, in the limit $q=t^{\alpha}, t \rightarrow 1$ the Hecke generators 
\eqref{hecke-gen} are again just simple transpositions $s_i$. It follows that
\begin{align*}
P_{\lambda}^{(\alpha)}(x_1,\dots,x_n)
=
\sum_{\sigma \in S_{\lambda} }
s_{\sigma}
\circ
x_{\lambda}
\circ
\prod_{i=1}^{r-1}
\left(
\sum_{\sigma \in S_{\lambda[i]}}
C_i
\left(
\begin{array}{@{}c@{}}
\lambda[i-1]
\medskip \\ 
\sigma\circ\lambda[i]
\end{array}
\right)
s_{\sigma}
\circ
x_{\lambda[i]}
\circ
\right)
1,
\end{align*}
where it is trivial to calculate the action of $s_{\sigma}$ at each step, since it simply acts on monomials in $(x_1,\dots,x_n)$ by the permutation $\sigma$. We therefore obtain
\begin{multline*}
P_{\lambda}^{(\alpha)}(x_1,\dots,x_n)
=
\sum_{\sigma[0] \in S_{\lambda[0]}}
\sum_{\sigma[1] \in S_{\lambda[1]}}
\cdots
\sum_{\sigma[\b{r}] \in S_{\lambda[\b{r}]}}
\prod_{i=1}^{\b{r}}
C_i
\left(
\begin{array}{@{}c@{}}
\lambda[i-1]
\medskip \\ 
\sigma[i]\circ\lambda[i]
\end{array}
\right)
\\
\times
\Big\{ \sigma[0] \circ x_{ \lambda[0]} \Big\}
\Big\{ \sigma[0] \circ \sigma[1] \circ x_{\lambda[1]} \Big\}
\cdots
\Big\{ \sigma[0] \circ \sigma[1] \circ \cdots \circ \sigma[\b{r}] \circ x_{\lambda[\b{r}]} \Big\},
\end{multline*}
where we have set $r-1 \equiv \b{r}$ to make the formula more compact.

\subsection{$q$--Whittaker polynomials}

Another special case of \eqref{main-formula} is when $t=0$, giving $q$--Whittaker polynomials $P_{\lambda}(x_1,\dots,x_n;q,0)$. Taking this specialization of 
\eqref{main-formula}, we obtain
\begin{align*}
P_{\lambda}(x_1,\dots,x_n;q,0)
=
\sum_{\sigma \in S_{\lambda} }
D_{\sigma}
\circ
x_{\lambda}
\circ
\prod_{i=1}^{r-1}
\left(
\sum_{\sigma \in S_{\lambda[i]}}
C_i
\left(
\begin{array}{@{}c@{}}
\lambda[i-1]
\medskip \\ 
\sigma\circ\lambda[i]
\end{array}
\right)
D_{\sigma}
\circ
x_{\lambda[i]}
\circ
\right)
1
\end{align*}
with coefficients that satisfy $C_i(\lambda,\mu) = 0$ if any $0 < \lambda_k < \mu_k$, and $C_i(\lambda,\mu)=\prod_{j = i+1}^{r} q^{(j-i)a_j(\lambda,\mu)}$ otherwise, and where each $D_{\sigma}$ is now composed of the divided-difference operators
\begin{align*}
D_{i}
=
(x_i/x_{i+1}-1)^{-1}(1-s_i),
\qquad
1 \leq i \leq n-1.
\end{align*}

\section*{Acknowledgments}

We thank the Galileo Galilei Institute and the organisers of the research program \textit{Statistical Mechanics, Integrability and Combinatorics} for kind hospitality during part of this work. It is a pleasure to thank Luigi Cantini for collaboration on \cite{CantinidGW}, which was very motivational for this work; Philippe Di Francesco for bringing the papers \cite{Kirillov_Noumi1,Kirillov_Noumi2} to our attention; Ole Warnaar for directing us to the references \cite{Duchamp_et_al,Hivert} for equation \eqref{HL1} and for many helpful remarks; and Paul Zinn-Justin for extended discussions on related topics. JdG and MW are generously supported by the Australian Research Council (ARC) and the ARC Centre of Excellence for Mathematical and Statistical Frontiers (ACEMS).

\appendix

\section{Explicit examples}

\subsection{Two variable example}

In the case $\lambda = (3,1)$, we have 
\begin{align*}
\lambda[0] = (3,1),\ \lambda[1] = (3,0),\ \lambda[2] = (3,0),\ \lambda[3] = (0,0),
\end{align*}
and $S_{\lambda[0]} = S_{\lambda[1]} = S_{\lambda[2]} = S_2$. Hence
\begin{align*}
P_{(3,1)}(x_1,x_2;q,t)
=
\sum_{\sigma \in S_2}
T_{\sigma}
\circ
x_1 x_2
\circ
\sum_{\rho \in S_2}
C_1\left( 
\begin{array}{@{}l@{}l@{}}
\lambda[0]_1 & \lambda[0]_2 \\
\lambda[1]_{\rho_1} & \lambda[1]_{\rho_2} 
\end{array}\right)
T_{\rho}
\circ
x_1
\circ
\sum_{\pi \in S_2}
C_2\left( 
\begin{array}{@{}l@{}l@{}}
\lambda[1]_1 & \lambda[1]_2 \\
\lambda[2]_{\pi_1} & \lambda[2]_{\pi_2} 
\end{array}\right)
T_{\pi}
\circ
x_1.
\end{align*}
We compute each sum in turn, starting with the rightmost:
\begin{align*}
\sum_{\pi \in S_2}
C_2\left( 
\begin{array}{@{}l@{}l@{}}
\lambda[1]_1 & \lambda[1]_2 \\
\lambda[2]_{\pi_1} & \lambda[2]_{\pi_2} 
\end{array}\right)
T_{\pi}
\circ
x_1
&=
x_1
+
C_2\left( 
\begin{array}{@{}cc@{}}
3 & 0 \\
0 & 3
\end{array}\right)
T_{1}
\circ
x_1
=
x_1
+
q \left( \frac{1-t}{1-qt} \right)
x_2.
\end{align*}
Combining with the middle sum, we find that
\begin{align*}
\sum_{\rho \in S_2}
C_1\left( 
\begin{array}{@{}l@{}l@{}}
\lambda[0]_1 & \lambda[0]_2 \\
\lambda[1]_{\rho_1} & \lambda[1]_{\rho_2} 
\end{array}\right)
T_{\rho}
\circ
\left(
x_1^2
+
q\frac{1-t}{1-qt} x_1 x_2
\right)
&=
x_1^2
+
q\frac{1-t}{1-qt} x_1 x_2,
\end{align*}
where we have used the fact that
$C_1\left( 
\begin{array}{@{}cc@{}}
3 & 1 \\
3 & 0
\end{array}\right)= 1$ and 
$C_1\left( 
\begin{array}{@{}cc@{}}
3 & 1 \\
0 & 3
\end{array}\right)= 0$. Combining everything and passing to the leftmost sum, we have
\begin{align*}
P_{(3,1)}(x_1,x_2;q,t)
&=
\sum_{\sigma \in S_2}
T_{\sigma}
\circ
\left(
x_1^3 x_2
+
q  \frac{1-t}{1-qt} 
x_1^2 x_2^2
\right)
\\
&=
(1+T_{1})\circ
\left(
x_1^3 x_2
+
q  \frac{1-t}{1-qt} 
x_1^2 x_2^2
\right)
=
x_1^3 x_2
+
\frac{1-t+q-qt}{1-qt}
x_1^2 x_2^2
+
x_1 x_2^3.
\end{align*}

\subsection{Three variable example}

In the case $\lambda = (3,2,1)$, we have
\begin{align*}
\lambda[0] = (3,2,1),\
\lambda[1] = (3,2,0),\
\lambda[2] = (3,0,0),\
\lambda[3] = (0,0,0),
\end{align*}
and $S_{\lambda[0]} =S_{\lambda[1]} = S_3$, $S_{\lambda[2]} = \{(1,2,3),(2,1,3),(2,3,1)\}$. Hence
\begin{multline*}
P_{(3,2,1)}(x_1,x_2,x_3;q,t)
=
\\
\sum_{\sigma \in S_3}
T_{\sigma}
\circ
x_1 x_2x_3
\circ
\sum_{\rho\in S_3}
C_1\left( 
\begin{array}{@{}l@{}l@{}l@{}}
\lambda[0]_1 & \lambda[0]_2 & \lambda[0]_3 \\
\lambda[1]_{\rho_1} & \lambda[1]_{\rho_2} & \lambda[1]_{\rho_3}
\end{array}\right)
T_{\rho}
\circ
x_1x_2
\circ
\sum_{\pi \in S_{\lambda[2]}}
C_2
\left(
\begin{array}{@{}l@{}l@{}l@{}}
\lambda[1]_1 & \lambda[1]_2 & \lambda[1]_3 \\
\lambda[2]_{\pi_1} & \lambda[2]_{\pi_2} & \lambda[2]_{\pi_3}
\end{array}\right)
T_{\pi}
\circ
x_1.
\end{multline*}
Starting with the rightmost sum, we have
\begin{align*}
\sum_{\pi \in S_{\lambda[2]}}
C_2
\left(
\begin{array}{@{}l@{}l@{}l@{}}
\lambda[1]_1 & \lambda[1]_2 & \lambda[1]_3 \\
\lambda[2]_{\pi_1} & \lambda[2]_{\pi_2} & \lambda[2]_{\pi_3}
\end{array}\right)
T_{\pi}
\circ
x_1
&=
x_1
+
C_2
\left(
\begin{array}{@{}ccc@{}}
3 & 2 & 0 \\
0 & 0 & 3
\end{array}\right)
T_{2} T_{1}
\circ
x_1
= 
x_1 +q\left(\frac{1-t}{1-qt^2}\right) x_3,
\end{align*}
where we have used that $C_2
\left(
\begin{array}{@{}ccc@{}}
3 & 2 & 0 \\
0 & 3 & 0
\end{array}\right)=0$. Combining with the middle sum and using the fact that 
$C_1
\left(
\begin{array}{@{}ccc@{}}
3 & 2 & 1 \\
3 & 2 & 0
\end{array}\right)=1
$
and all other $C_1=0$, we find
\begin{align*}
\sum_{\rho\in S_3}
C_1\left( 
\begin{array}{@{}l@{}l@{}l@{}}
\lambda[0]_1 & \lambda[0]_2 & \lambda[0]_3 \\
\lambda[1]_{\rho_1} & \lambda[1]_{\rho_2} & \lambda[1]_{\rho_3}
\end{array}\right)
T_{\rho}
\circ \left(  x_1^2x_2 +q\frac{1-t}{1-qt^2}x_1x_2 x_3\right)
=  x_1^2x_2 +q\frac{1-t}{1-qt^2}x_1x_2 x_3 .
\end{align*}
Combining everything, the leftmost sum finally gives
\begin{align*}
\sum_{\sigma \in S_3}
T_{\sigma} \circ \left(  x_1^3 x_2^2 x_3 +q\frac{1-t}{1-qt^2}x_1^2 x_2^2 x_3^2 \right) = \sum_{\sigma \in S_3} x_{\sigma_1}^3 x_{\sigma_2}^2 x_{\sigma_3} +  (2+q+t+2qt)\frac{1-t}{1-qt^2}x_1^2x_2^2 x_3^2.
\end{align*}


\begin{thebibliography}{99}
\bibitem{CantinidGW} L.~Cantini, J.~de~Gier and M.~Wheeler, \textit{Matrix product formula for Macdonald polynomials}, arXiv:1505.00287.

\bibitem{Cher95a} I. Cherednik, \textit{Double affine Hecke algebras and Macdonald's conjectures}, Annals Math. \textbf{141} (1995), 191--216.

\bibitem{Cher95b} I. Cherednik, \textit{Nonsymmetric Macdonald polynomials}, Internat. Math. Res. Notices \textbf{10} (1995), 483--515.

\bibitem{Duchamp_et_al} G. Duchamp, D. Krob, A. Lascoux, B. Leclerc, T. Scharf and J. Y. Thibon, \textit{Euler--Poincar\'e characteristic and polynomial representations of Iwahori--Hecke algebras}, Publ. RIMS \textbf{31} (1995), 179--201.

\bibitem{HaglundHL1} J. Haglund, M. Haiman and N. Loehr, \textit{A combinatorial formula for Macdonald polynomials}, J. Amer. Math. Soc. \textbf{18} (2005), 735--761; arXiv:math/0409538.

\bibitem{HaglundHL2} J. Haglund, M. Haiman and N. Loehr, \textit{A combinatorial formula for non-symmetric Macdonald polynomials}, Amer. J. Math. \textbf{130} (2008), 359--383; arXiv:math/0601693.

\bibitem{Hivert} F. Hivert, \textit{Hecke algebras, difference operators and quasi-symmetric functions}, Adv. Math. \textbf{155} (2000), 181--238.

\bibitem{Kirillov_Noumi1} A. N. Kirillov and M. Noumi, \textit{$q$-difference raising operators for Macdonald polynomials and the integrality of transition coefficients}, Algebraic Methods and $q$-Special Functions, CRM Proceedings and Lecture Notes, Vol. {\bf 22} (1999); arXiv:q-alg/9605005.

\bibitem{Kirillov_Noumi2} A. N. Kirillov and M. Noumi, \textit{Affine Hecke algebras and raising operators for Macdonald polynomials}, arXiv:q-alg/9605004.

\bibitem{macd88} I. Macdonald, \textit{A new class of symmetric functions}, Publ. I.R.M.A. Strasbourg, Actes 20$^\textrm{e}$ S\'eminaire Lotharingien \textbf{131-71} (1988).

\bibitem{MacdBook} I. Macdonald, \textit{Symmetric functions and Hall polynomials}, (2nd ed.), Oxford, Clarendon Press 1995.

\bibitem{Opdam} E. Opdam, \textit{Harmonic analysis for certain representations of graded Hecke algebras}, Acta Math. \textbf{175} (1995), 75--121.

\bibitem{RamY} A. Ram and M. Yip, \textit{A combinatorial formula for Macdonald polynomials}, Adv. Math. \textbf{226} (2011), 309--331; arXiv:0803.1146.

\end{thebibliography}
\end{document}